\documentclass[12pt]{amsart}

\usepackage{mystyle}

\begin{document}

\title{Further Time Regularity for Fully Non-Linear Parabolic Equations}

\author[H. A. Chang-Lara]{H\'ector A. Chang-Lara}
\address{Department of Mathematics, Columbia University, New York, NY 10027}
\email{changlara@math.columbia.edu}

\author[D. Kriventsov]{Dennis Kriventsov}
\address{Department of Mathematics, University of Texas, Austin, TX 78712}
\email{dkriventsov@math.utexas.edu}

\begin{abstract}
We establish H\"older estimates for the time derivative of solutions of fully non-linear parabolic equations that does not necessarily have $C^{2,\bar\a}$ estimates.
\end{abstract}
\subjclass{35B45, 35B65, 35K10, 35K55}
\keywords{Further regularity in time, H\"older estimates, Krylov-Safanov}

\maketitle


\section{Introduction}

We are interested in studying further regularity in time for non-homogeneous parabolic equations of the form
\begin{align*}
u_t-F(D^2u,Du,x) = f(x,t) \text{ in } B_1\times(-1,0].
\end{align*}
Our estimates do not assume that $f$ is differentiable, nor that the homogeneous problem with frozen coefficients has interior $C^{2,\bar\a}$ estimates for which the H\"older continuity of $u_t$ is already well known. Let us now recall how these cases can be treated.

We can get an estimate for $u_t$ by considering the equation obtained by taking the time derivative of the original problem. However, this relies on $f(x,\cdot) \in C^{0,1}$. Approximation techniques can be used when $f \in C^{\bar\gamma}$ and the homogeneous problem has $C^{2,\bar\a}$ estimates, see \cite{Wang92-2} or Chapter 8 from \cite{Caffarelli95}. This actually implies $u \in C^{2,\gamma}$ for every $\gamma \in (0,\bar\a)\cap(0,\bar\gamma]$, therefore by the scaling of the equation, $u(x,\cdot) \in C^{1,\gamma/2}$. Finally, if only a $C^{1,\bar\a}$ estimate is available for the homogeneous problem, then a similar approach proves that $u(x,\cdot) \in C^{0,(1+\gamma)/2}$, see \cite{Wang92-2,Caffarelli95}. H\"older estimates for $u_t$, without assuming smooth coefficients or $C^{2,\bar\a}$ estimates, seem to be unknown up to this moment.

Given that $f \in C^{0,\gamma}$, the scaling of the problem suggests that $u_t \in C^{0,\gamma}$ as well. Our main theorem establishes such an estimate for $\gamma\in(0,\bar\a)$ where $\bar\a\in(0,1)$ is the universal H\"older exponent from the Krylov-Safonov estimate.

\begin{theorem}
Let $F$ be uniformly elliptic satisfying hypothesis \eqref{eq:H} (defined in Section \ref{preliminaries}) with $F(0,0,x) = 0$ and $u$ satisfies,
\begin{align*}
u_t - F(D^2u,Du,x) = f(x,t) \text{ in the viscosity sense in } Q_1
\end{align*}
Assume that $f \in C^{0,\gamma}(Q_1)$ for some $\gamma \in (0, \bar \a)$ where $\bar \a$ is the exponent from the Krylov-Safonov theorem. Then $u_t$ exists pointwise, and for some constant $C>0$ depending on $\min(\gamma,(\bar\a-\gamma))$,
\[
 \|u_t\|_{C^{0,\gamma}(Q_{1/4})} \leq C\1\sup_{Q_1}|u| + \sup_{x\in B_1}\|f(x,\cdot)\|_{C^{0,\gamma/2}[-1,0]}\2.
\]
\end{theorem}

Notice that the scaling of this estimate corresponds to the scaling for $u_t$ which formally satisfies an equation with the singular right-hand side given by $f_t$. This already suggests that we should establish a diminish of oscillation for $u_t$ and not necessarily $u$. Keep in mind that a diminish of oscillation for $u$, leading to a similar conclusion would imply that $u$ has second derivatives in space which is known to be false for a general $F$, see for instance the celebrated counterexamples by N. Nadirashvili and S. Vl{\u{a}}du{\c{t}} in \cite{Nadirashvili-2013} and the references therein.

As $u_t$ is not controlled a priori, we establish a diminish of oscillation for the difference quotients $\frac{\d_\t u(x,t)}{\t^\b} := \frac{u(x,t)-u(x,t-\t)}{\t^\b}$ where $\b\in(0,1)$. This allows us to control in each step a higher order difference quotient and, after a finite number of iterations, obtain the desired regularity for $u_t$ as is done in Chapter 5 of \cite{Caffarelli95}.

To show a diminish of oscillation for $\frac{\d_\t u}{\t^\b}$ brings several challenges. First of all the equation for $\frac{\d_\t u}{\t^\b}$ has a right-hand side that might degenerate as $\t$ approaches zero. On the other hand, by using the scaling for $\frac{\d_\t u}{\t^\b}$ we make $u$ grow. The key idea is to assume some small a priori H\"older continuity for the difference quotient which gives a way to control the difference quotients for $\t \in (0,\bar\t)$ by the difference quotients with $\t>\bar\t$. This is rigorously established in the proof of Lemma \ref{thm:improvement}. 

Moreover, we also obtain a $C^{0,\gamma}$ estimate for $\frac{\d_\t u}{\t^\gamma}$ for any $\gamma\in(0,1)$ by only assuming $f$ bounded. This is actually the first step for the proof of our main theorem contained in Section \ref{further}.

Finally we would like to point out that the constant in our estimate degenerates as $\gamma$ approaches zero or $\bar\a$. Whether this estimate can be improved or the existence of counterexamples remains open. Another interesting question is whether the result can be extended for $F$ depending on the time variable. Notice that for $F = F(M,t)$ uniformly elliptic,
\begin{align*}
\frac{d}{dt}F(D^2u,t) = F_M(D^2u,t)D^2u_t + F_t(D^2u,t).
\end{align*}
The first term is the one that can be used in the linearized equation. By uniform ellipticity $F_M(D^2u,t)$ is bounded from above and below away from zero. The second term is a bad one, keep in mind that without assuming $C^{1,1}$ estimates $F_t(D^2u,t)$ might be unbounded. For example, consider
\[
F(D^2u,t) = \sup_\a\inf_\b\trace(A_{\a,\b}(t)D^2u).
\]

\subsection{Applications to fully non-linear, integro-differential, parabolic equations}

The main interest for the authors to study this problem comes from fully non-linear, integro-differential, parabolic equations; let us recall first a singular counterexample concerning a time regularity issue. Given $\s\in(0,2)$, is it known from \cite{Davila12-p} that there exists some Dirichlet data in $(B_1^c\times(-1,0])\cup (\R^n\times\{-1\})$ such that the solution to the fractional heat equation $u_t-\D^{\s/2} u=0$ in $B_1\times(-1,0]$ can not be smoother than Lipschitz continuous in the time variable. This is surprising, as it is well known that for $\s=2$ the solutions are smooth.

H\"older estimates for fully nonlinear and non local parabolic problems were established by G. D\'avila and the first author in \cite{Davila12-p}. Further regularity estimates, as the analogue of the parabolic Evans-Krylov theorem \cite{evans1982classical, krylov1983boundedly, caffarelli2010evans, MR2831115}, seem to require either better time estimates for the solution or strong hypothesis on the data as was considered in \cite{2014arXiv1408.5149L}.

The authors plan to investigate further regularity in time for fully non-linear, integro-differential, parabolic equations. We expect that a H\"older modulus of continuity for the boundary data, just in time, makes the time derivative of the solution H\"older continuous in space and time. For example, by truncating the solution of a homogeneous fractional heat equation we can transfer the Dirichlet data to a right-hand side. Notice that the smoothness of the kernel associated with $\D^{\s/2}$ (outside of the origin) has only a regularizing effect in space but not in time. At this point it is not difficult to see that the time derivative has a modulus of continuity in space and time if the Dirichlet data had a modulus of continuity in time, at least for the heat equation.

Let us take the opportunity at this point to mention that a similar phenomenon was found by J. Serra in \cite{Serra14-2} for the $C^{k,\bar\a'}$ ($k=\lfloor\s+\bar\a\rfloor, \ \bar\a'=k-(\s+\bar\a) \neq 0$) estimates of concave non-local equations of order $\s$ with rough kernels. In this work the Dirichlet data is assumed to be $C^{0,\bar\a}$; moreover, it proves that the $C^{k,\bar\a'}$ estimate is false without this assumption by giving a counterexample. This technique was first introduced by the same author in \cite{Serra14} for the parabolic setting were it concludes that $u \in C^{1,\bar\a}$, which was known only for the case of smooth kernels. Thanks to the scaling, $u_t(x,\cdot) \in C^{(1+\bar\a)/\s} \ss C^{1,\bar\a'}$ if $\s\in(0,1]$. Notice that as $\s\to2$ the estimate leaves a gap between $C^{0,(1+\bar\a)/2} \ss C^{0,1/2}$ and the well known $C^{1,\bar\a}$ regularity in time.

\section{Preliminaries.}\label{preliminaries}

We use the following notation, which is standard for second order parabolic problems. Given $\W\subset\R^n, \ A\subset\R^n\times\R$ and $\a,\t\in(0,1)$,
\begin{align*}
\p_p\1\W \times (t_1,t_2]\2 &:= \1\W \times \{t_1\}\2 \cup \1\p\W\times(t_1,t_2]\2,\\
[u]_{C^\a(A)} &:= \sup_{(x,t),(x',t') \in A} \frac{|u(x,t)-u(x',t')|}{(|x-x'| + |t-t'|^{1/2})^\a},\\
\d_\t u(x,t) &:= u(x,t) - u(x,t-\t).
\end{align*}
We will frequently use the parabolic cylinders $Q_r(x,t)=B_r(x)\times (t-r^2,t)$. The cylinder $Q_r$ is centered on $(0,0)$.


Let $\mathcal S \ss \R^{n\times n}$ be the space of $n$ by $n$ symmetric matrices and $I$ its identity matrix. A continuous function $F:\mathcal S \times \R^n \times \W \to\R$ is said to be uniformly elliptic with respect to $0<\l\leq\L<\8$ if for each $x \in \R^n$,
\begin{equation}
\cM^-(M-N) - \L|p-q| \leq F(M,p,x)-F(N,q,x) \leq \cM^+(M-N) + \L|p-q|, 
\end{equation}
where
\begin{align*}
\cM^+M &:= \sup\{\trace(AM): A \in \mathcal S, \l I\leq A \leq \L I\},\\
\cM^-M &:= \inf\{\trace(AM): A \in \mathcal S, \l I\leq A \leq \L I\}.
\end{align*}

Any constant that depends on $n, \ \l$ and $\L$ is considered universal. The dependence of various values on these quantities will be assumed without being stated explicitly.

Solutions are considered in the viscosity sense as in \cite{Wang92,Caffarelli95}. Another good reference is the lecture notes by C. Imbert and  L. Silvestre, \cite{Silvestre-Imbert-2013}. This notion is sufficiently weak to allow for existence of continuous solutions to the Dirichlet problem by Perron's method.

For smooth functions $u$ and $v$ satisfying,
\begin{align*}
\label{eq:H}
u_t - F(D^2u,Du,x) = f(x,t) \qquad\text{and}\qquad v_t - F(D^2v,Dv,x) = g(x,t),
\end{align*}
the uniform ellipticity of $F$ implies for $w:=u-v$ the following inequalities,
\[
w_t - \cM^+(D^2w) - \L|Dw| \leq f(x,t) - g(x,t) \leq w_t - \cM^-(D^2w) + \L|Dw|.
\]
However, the question of whether two viscosity solutions imply the same inequalities in the viscosity sense is a delicate one. A sufficient condition is that $F$ satisfies the Lipschitz estimate
\begin{align}
\label{eq:H}\tag{H}
 |F(M,p,x)-F(M,p,y)|\leq C |x-y|(1 + |M|+|p|).
\end{align}
See \cite{Jensen89,Crandall92} for a rather comprehensive discussion and many references. It is important to note that any regularity of $F$ needed to ensure this uniform ellipticity identity in the viscosity sense is used only qualitatively, and there is no dependence on it in our estimates. 

The following basic interior regularity estimate is a consequence of the Krylov-Safonov Harnack inequality.

\begin{theorem}[Krylov-Safonov]\label{thm:KS}
There exists a universal exponent $\bar\a\in (0,1)$ and constant $C$ such that for $u$ satisfying in $Q_1$
\begin{align}\label{eq:basic}
u_t - \cM^+ u - \L|Du| &\leq |f| \quad \text{ and } \quad u_t - \cM^- u + \L|Du| \geq -|f|,
\end{align}
in the viscosity sense, then
\begin{align*}
\|u\|_{C^{\bar\a}\1 Q_{1/2}\2} \leq C\1\sup_{Q_1}|u| +\|f\|_{L^{n+1}(Q_1)}\2.
\end{align*}
\end{theorem}

From now on we fix $\bar\a \in (0,1)$ to be the exponent in Theorem \ref{thm:KS}.



\section{Bounded Right-Hand Side}\label{further}

The goal of this section is to show that under the assumption that $f$ is bounded, we have that $u$ is in every H\"older space with exponent $\gamma\in(0,1)$. The argument will proceed iteratively, with each step proving that the difference quotients $\frac{\d_\t u}{\t^\b}$ are bounded for progressively higher $\b$. The crucial step is to control $\frac{\d_\t u}{\t^{\b+\a/2}}$ given that $\frac{\d_\t u}{\t^\b}$ is already controlled ($\a,\b\in(0,1)$ and $\b+\a/2\in(0,1)$), similar to what is done in Chapter 5 from \cite{Caffarelli95} in order to prove $C^{1,\bar\a}$ estimates. It turns out to be useful to control quantities like $\|\frac{\d_\t u}{\t^\b}\|_{C^{0,\e}}$ for some small $\e$, as this allows to have control on the difference quotients for $\t$ approaching zero by using Corollary \ref{lem:appendix3} in the appendix. This entire section is devoted to proving the following theorem:

\begin{theorem}\label{thm:improv_reg_1}
Let $F$ be uniformly elliptic satisfying hypothesis \eqref{eq:H} with $F(0,0,x) = 0$ and $u$ satisfies,
\begin{align*}
u_t - F(D^2u,Du,x) = f(x,t) \text{ in the viscosity sense in } Q_2
\end{align*}
Given $\gamma\in(0,1)$ there exists $\e\in(0,1-\gamma)$ and $C>0$ depending on $(1-\gamma)$ such that,
\[
\sup_{\t\in(0,1/4)} \left\|\frac{\d_\t u}{\t^\gamma}\right\|_{C^\e(Q_{1/2})} \leq C\1\osc_{Q_2}u+\sup_{Q_2}|f|\2.
\]
\end{theorem}

The key step is stablished in the following Lemma. Notice that for $\e=0$ the statement is a diminish of oscillation leading to a $C^{0,\a}$ estimate for the difference quotient.

\begin{lemma}\label{thm:improvement}
Let $F$ be uniformly elliptic satisfying hypothesis \eqref{eq:H} and $u$ satisfies,
\begin{align*}
&u_t-F(D^2 u,Du,x)=f(x,t) \text{ in the viscosity sense in } Q_2,
\end{align*}
Let $\b\in(0,1), \a\in(0,\bar\a)$ and $\e\in(0,\min(1-\b,\bar\a-\a))$ where $\bar\a$ is the exponent from Krylov-Safonov estimates. Then there exists constants $\m,\d\in(0,1)$ depending on $\a$ and $\e$, such that,
\[
\sup_{\t\in(0,1)} \left[\frac{\d_{\t} u}{\t^\b}\right]_{C^{0,\e}(Q_1)}\leq 1 \qquad \text{ and } \qquad \sup_{\substack{\t\in(0,1)\\(x,t)\ss Q_1}}|\d_\t f(x,t)| \leq \d,
\]
imply,
\[
\sup_{\t\in(0,1)}\left[\frac{\d_{\t} u}{\t^\b}\right]_{C^{0,\e}(Q_\m)}\leq \m^\a.
\]
\end{lemma}

\begin{proof}
The value $\mu\in(0,1)$ will remain fixed for the duration of the proof; it will be specified later explicitly. Assume by contradiction that for $\d \in(0,1)$, there exists $F$ and $u$ such that,
\begin{align*}
&u_t-F(D^2 u,Du,x)=f(x,t) \text{ in the viscosity sense in } Q_2,\\
&\sup_{\t\in (0,1)} \left[\frac{\d_\t u}{\t^\b}\right]_{C^{0,\e}(Q_1)}\leq 1, \qquad \text{ and } \qquad  \sup_{\substack{\t\in(0,1)\\(x,t)\in Q_1}}|\d_\t f(x,t)| \leq \d,
\end{align*}
however, there exists a cylinder $Q_{r}(x_0,t_0) \ss Q_\m$ for which,
\[
 \sup_{\t\in (0,1)}\osc_{Q_{r}(x_0,t_0)}\frac{\d_\t u}{\t^\b}>\mu^\a r^\e.
\]

Consider the following rescaling for $\kappa := r/\mu$,
\[
w(x,t) := \kappa^{-(2\b+\e)}u\1\kappa x+x_0,\kappa^2 t+t_0\2.
\]
It satisfies,
\[
w_t(x,t)-\tilde{F}(D^2 w,Dw,x)=\tilde{f}(x,t) \text{ in the viscosity sense in } Q_2
\]
where,
\begin{align}
\nonumber
&\tilde{F}(M,p,x):=\kappa^{2-2\b-\e}F\1\kappa^{-(2-2\b-\e)}M,\kappa^{-(1-2\b-\e)}p,\kappa x+x_0\2,\\
\nonumber
&\tilde{f}(x,t):= \kappa^{2-2\b-\e}f\1\kappa x+x_0,\kappa^2 t+t_0\2.
\end{align}
Notice that the hypothesis for $f$ implies the following for $\tilde f$ provided that $\b\in(0,1-\e/2)$,
\begin{align}
\label{eq:rescaledrhs}
&\sup_{\substack{\t\in(0,\kappa^{-2})\\(x,t)\in Q_1}}|\d_\t \tilde{f}(x,t)| \leq \d.
\end{align}
The hypotheses for the difference quotients of $u$ imply that
\begin{align}
\label{eq:rescaledupbd}
&\sup_{\t\in(0,\kappa^{-2})}\left[\frac{\d_\t w}{\t^\b}\right]_{C^{0,\e}(Q_1)}\leq \sup_{\t\in(0,1)}\left[\frac{\d_\t u}{\t^\b}\right]_{C^{0,\e}(Q_1)}\leq 1,\\
\label{eq:contradiction}
&\sup_{\t\in(0,\kappa^{-2})}\osc_{Q_{\mu}}\frac{\d_\t w}{\t^\b}=\kappa^{-\e}\sup_{\t\in(0,1)}\osc_{Q_{r}(x_0,t_0)}\frac{\d_\t u}{\t^\b}>\mu^{\a+\e}.
\end{align}

The next step consists on showing that a hypothesis similar to \eqref{eq:contradiction} holds taking the supremum with respect to $\tau$ away from zero. Namely $\t\in(\bar\t,\kappa^{-2})$ for some $\bar\t\in(0,\k^{-2})$ depending on $\m$ and $\e$. Indeed, define for $(x,t),(y,s)\in Q_\m$,
\[
z(a)=w\1x,t+\kappa^{-2}a\2-w\1y,s+\kappa^{-2}a\2,
\]
applying Corollary \ref{lem:appendix3} to $z$:
\begin{align*}
&\sup_{\t\in(\bar\t,\kappa^{-2})} \left|\frac{\d_\t w(x,t)}{\t^{\b}}-\frac{\d_\t w(y,s)}{\t^{\b}}\right|,\\
\geq &\frac{1}{2}\sup_{\t\in(0,\kappa^{-2})}\left|\frac{\d_\t w(x,t)}{\t^{\b}}-\frac{\d_\t w(y,s)}{\t^{\b}}\right| - C \bar\t^{\e/2}\sup_{\t\in(0,\kappa^{-2})}\left[\frac{\d_\t w}{\t^\b} \right]_{C^{0,\e}(Q_1)}.
\end{align*}
The second term on the right-hand side is controlled from \eqref{eq:rescaledupbd}. After taking the supremum in $(x,t),(y,s)\in Q_\m$ and using \eqref{eq:contradiction}, this gives
\begin{align}
\label{eq:contradiction2}
\sup_{\t\in(\bar\t,\kappa^{-2})} \osc_{Q_\m}\frac{\d_\t w}{\t^\b} \geq \frac{\m^{\a+\e}}{2}-C\bar\t^{\e/2} \geq \frac{\m^{\a+\e}}{4},
\end{align}
provided that $\bar\t^{\e/2}$ is sufficiently small with respect to $\m^{\a+\e}$.

Let $\t\in(\bar\t,\kappa^{-2})$ and
\[
v(x,t) = \frac{\d_\t w}{\t^\b}(x,t) - \frac{\d_\t w}{\t^\b}(0,0)
\]
By time translation invariance and the hypothesis \eqref{eq:H} we get that $v$ satisfies two viscosity inequalities in $Q_1$,
\begin{align*}
 v_t - \cM^+(D^2 v) -\L|Dv| \leq \frac{\d_\t \tilde{f}}{\t^\b} \qquad \text{ and } \qquad v_t - \cM^-(D^2 v) + \L|Dv| \geq \frac{\d_\t \tilde{f}}{\t^\b}.
\end{align*}
In order to apply the estimate in the Krylov-Safanov Theorem \ref{thm:KS} we need to control the two terms on the right hand side. Using \eqref{eq:rescaledupbd} we get that $\osc_{Q_1} v \leq 1$. The right-hand side gets also controlled by one provided we take $\d\in(0, \bar\t^\b)$,
\[
\sup_{Q_1}\left|\frac{\d_\t \tilde{f}}{\t^\b}\right| \leq \frac{\d}{\t^\b} \leq 1.
\]
Finally, by the H\"older estimate in Theorem \ref{thm:KS}, and fixing now $\m$ sufficiently small in terms of $(\bar\a-(\a+\e))$, we get the following contradiction to \eqref{eq:contradiction2},
\[
 \osc_{Q_\m} v \leq C\m^{\bar\a}\leq \frac{\m^{\a+\e}}{8}.
\]
\end{proof}

Notice that the previous lemma is independent on the size of the oscillation of the solution $u$. This allows us to prove the following corollary by considering an appropriated rescaling for $\frac{\d_\t u}{\t^\b}$. The fact that the oscillation of $u$ increases by this rescaling is actually harmless.

\begin{corollary}\label{cor:iteration}
Let $F$ be uniformly elliptic satisfying hypothesis \eqref{eq:H} and $u$ satisfies,
\begin{align*}
&u_t-F(D^2 u,Du,x)=f(x,t) \text{ in the viscosity sense in } Q_2,
\end{align*}
Let $\b\in(0,1), \a\in(0,\min(1-\b,\bar\a))$ and $\e\in(0,\min(1-(\b+\a/2),\bar\a-\a))$ where $\bar\a$ is the exponent from Krylov-Safonov estimates. Then there exists constants $\m,\d\in(0,1)$ depending on $\a$ and $\e$, such that,
\[
\sup_{\t\in(0,1)} \left[\frac{\d_{\t} u}{\t^\b}\right]_{C^{0,\e}(Q_1)}\leq 1 \qquad \text{ and } \qquad \sup_{\substack{\t\in(0,1)\\(x,t)\ss Q_1}}|\d_\t f(x,t)| \leq \d,
\]
imply for every $i\in\N$,
\[
\sup_{\t\in(0,\m^{2i})}\left[\frac{\d_{\t} u}{\t^\b}\right]_{C^{0,\e}\1 Q_{\m^i}\2}\leq \m^{\a i}.
\]
\end{corollary}

\begin{proof}
Assume for some $i\in\N$ the inductive hypothesis,
\[
\sup_{\t\in(0,\m^{2i})}\left[\frac{\d_{\t} u}{\t^\b}\right]_{C^{0,\e}\1 Q_{\m^i}\2}\leq \m^{\a i}.
\]

Let
\begin{align*}
v(x,t) := \frac{u(\m^i x,\m^{2i}t)}{\m^{2i(\b+\a/2+\e/2)}}.
\end{align*}
such that,
\begin{align*}
&v_t-\tilde F(D^2 v,Dv,x)=\tilde f(x,t) \text{ in the viscosity sense in } Q_2,
\end{align*}
where,
\begin{align*}
\tilde F(M,p,x) &:= \m^{2i(1-\b-\a/2-\e/2)}F(\m^{-2i(1-\b-\a/2-\e/2)} M,\m^{-2i(1/2-\b-\a/2-\e/2)}p,\m^i x),\\
\tilde f(x,t) &:= \m^{2i(1-\b-\a/2-\e/2)}f(\m^i x,\m^{2i} t).
\end{align*}
Given that $\b+\a/2+\e/2<1$ we obtain that,
\begin{align}\label{eq:6}
\sup_{\substack{\t\in(0,1)\\(x,t)\ss Q_1}}|\d_\t \tilde f(x,t)| \leq \d.
\end{align}
Moreover, the inductive hypothesis tells us that,
\[
\sup_{\t\in(0,1)} \left[\frac{\d_{\t} v}{\t^\b}\right]_{C^{0,\e}(Q_1)} = \m^{-\a i}\sup_{\t\in(0,\m^{2i})} \left[\frac{\d_{\t} u}{\t^\b}\right]_{C^{0,\e}\1 Q_{\m^i}\2} \leq 1.
\]
By applying Lemma \ref{thm:improvement} to $v$ we now obtain the desired inductive step and the proof the corollary.
\end{proof}

The next corollary establishes an estimate over a higher order difference quotient by sacrificing a little bit of the $\e$ H\"older exponent.

\begin{corollary}\label{cor:improvement_dif_quot}
Let $F$ be uniformly elliptic satisfying hypothesis \eqref{eq:H} with $F(0,0,x) = 0$ and $u$ satisfies,
\begin{align*}
&u_t-F(D^2 u,Du,x)=f(x,t) \text{ in the viscosity sense in } Q_2,
\end{align*}
Let $\b\in(0,1), \a\in(0,\min(2-2\b,\bar\a))$ and $\e\in(0,\min(1-(\b+\a/2),\bar\a-\a))$ where $\bar\a$ is the exponent from Krylov-Safonov estimates. Then there exists constants $\m,\d\in(0,1)$ depending on $\a$ and $\e$, such that,
\[
\osc_{Q_2} u + \sup_{\t\in(0,1)} \left[\frac{\d_{\t} u}{\t^\b}\right]_{C^{0,\e}(Q_1)}\leq 1 \qquad \text{ and } \qquad \sup_{Q_2}|f| \leq \d,
\]
imply,
\[
\sup_{\t\in(0,1/4)}\left[\frac{\d_{\t} u}{\t^{\b+\a/2}}\right]_{C^{0,\bar\a\e/4}\1 Q_{1/2}\2}\leq C.
\]
\end{corollary}

\begin{proof}
By a standard covering argument applied to Corollary \ref{cor:iteration} we know that there exists $C>0$ such that,
\begin{align}\label{eq:2}
\sup_{\substack{Q_r(x,t)\ss Q_{1/2}\\\t\in(0,r^2)}}\osc_{Q_r(x,t)}\frac{\d_\t u}{r^{\a+\e}\t^\b}\leq C.
\end{align}
Our goal is to bound instead,
\begin{align*}
\sup_{\substack{Q_r(x,t)\ss Q_{1/2}\\\t\in(0,1/4)}}\osc_{Q_r(x,t)}\frac{\d_\t u}{r^{\bar\a\e/4}\t^{\b+\a/2}}.
\end{align*}

Let $x\in B_{1/2}$ and $v(t)=u(x,t)$. By hypothesis,
\[
\osc_{(-1/4,0]}v \leq \osc_{Q_2} u \leq 1.
\]
On the other hand, using \eqref{eq:2} with $\t=r^2$,
\begin{align*}
\left|\frac{\d^2_\t v(t)}{\t^{\b+\a/2+\e/2}}\right| \leq \osc_{Q_{\t^{1/2}}(x,t)} \frac{\d_\t u}{\t^{\b+\a/2+\e/2}} \leq C.
\end{align*}
We can apply now to $v$ the proof of Lemma 5.2 from \cite{Caffarelli95} using that $\b+\a/2+\e/2<1$ in order to obtain the following estimate independent of $x \in B_{1/2}$,
\begin{align*}
\sup_{\t\in(0,1/4)}\osc_{(-1/4,0]} \frac{\d_\t v}{\t^{\b+\a/2+\e/2}} \leq C.
\end{align*}
In particular, by the triangular inequality,
\begin{align*}
\sup_{\t\in(0,1/4)}\osc_{Q_{1/2}} \frac{\d_\t u}{\t^{\b+\a/2+\e/2}} \leq C.
\end{align*}

We now fix $Q_r(x,t)\ss Q_{1/2}$ and consider two cases. If $\t\in(0,r^{\bar\a/2})$ then from the previous estimate,
\begin{align*}
\osc_{Q_r(x,t)}\frac{\d_\t u}{r^{\bar\a\e/4}\t^{\b+\a/2}} \leq C\frac{\t^{\e/2}}{r^{\bar\a\e/4}}\leq C.
\end{align*}
On the other hand, if $\t\in[r^{\bar\a/2},1/4)$ then we use the Krylov-Safanov estimate \ref{thm:KS} and the fact that $2>\b+\a/2+\e/2$,
\begin{align*}
\osc_{Q_r(x,t)}\frac{\d_\t u}{r^{\bar\a\e/4}\t^{\b+\a/2}} \leq C\frac{r^{\bar\a(1-\e/4)}}{\t^{\b+\a/2}} \leq C.
\end{align*}
\end{proof}

By iterating Corollary \ref{cor:improvement_dif_quot} we get to control higher order difference quotients of the solution. This is the same approach found to address the H\"older estimates for the derivatives of a solution that satisfies a translation invariant equation, see Chapter 5 from \cite{Caffarelli95}. The consequence is the proof of Theorem \ref{thm:improv_reg_1}.

\begin{proof}[Proof of Theorem \ref{thm:improv_reg_1}]
Let us assume without loss of generality that for $\d\in(0,1)$ sufficiently small,
\begin{align*}
\osc_{Q_2} u \leq 1 \qquad\text{and}\qquad \sup_{Q_2}|f| \leq \d.
\end{align*}
Consider $\b_0 \in(0,3\bar\a/8)$, $\a:=\bar\a/2$ and $\b_k := \b_0 + k\a/2$. Our goal is to prove that as long $\b_k<1$ there exists some $\e_k\in(0,1)$ such that,
\[
\sup_{\t\in (0,2^{-2k})} \left[\frac{\d_\t u}{\t^{\b_k}}\right]_{C^{0,\e_k}\1Q_{2^{-k}}\2} \leq C(k).
\]
Then the result follows by a standard covering argument both for the domain of the equation and the interval of H\"older exponents. Notice also that $\b_k<1$ implies $k<4/\bar\a$ such that any dependence on $k$ is actually universal.

The first step is to establish the following bound for $\e\in(0,\bar\a/8)$ by using the Krylov-Safanov Theorem \ref{thm:KS},
\begin{align*}
\sup_{\t\in(0,1/4)}\left[\frac{\d_\t u}{\t^{\b_0}}\right]_{C^{\e}(Q_{1/2})} \leq C.
\end{align*}
Same as in the proof of Corollary \ref{cor:improvement_dif_quot} we consider two cases. If $\t\in(0,r^2)$ then we bound the oscillation in the time variable in terms of $\t$ by the Krylov-Safanov estimate and then use the triangular inequality to obtain,
\begin{align*}
\sup_{Q_r(x,t)\ss Q_{1/2}} \osc_{Q_r(x,t)} \frac{\d_\t u}{r^{\e} \t^{\b_0}} \leq C\frac{\t^{\bar\a/2-\b_0}}{r^{\e}} \leq C.
\end{align*}
If $\t\in[r^2,1/4)$ then we use instead that $\osc_{Q_r(x,t)} \d_\t u \leq \osc_{Q_r(x,t)} u + \osc_{Q_r(x,t-\t)} u$, for which each term gets controlled in terms of $r$, once again using the Krylov-Safanov estimate,
\begin{align*}
\sup_{Q_r(x,t)\ss Q_{1/2}} \osc_{Q_r(x,t)} \frac{\d_\t u}{r^{\e} \t^{\b_0}} \leq C\frac{r^{\bar\a-\e}}{\t^{\b_0}} \leq C.
\end{align*}

At this point we plan to iterate Corollary \ref{cor:improvement_dif_quot} $k$ times for some $k$ such that $\b_k = \b_0+k\a/2=\b_0+k\bar\a/4<1$. Let $\e_0 := \min(1-\b_k,\bar\a/16)$ and $\e_k := \e_0(\bar\a/4)^k$. Then, as long as $\b_k+\e_{k-1}/2<1$ we get that,
\[
\sup_{\t\in (0,2^{-2k})} \left[\frac{\d_\t u}{\t^{\b_k}}\right]_{C^{0,\e_k}\1Q_{2^{-k}}\2} \leq C(k).
\]
This establishes the desired estimate with constants that depend on $(1-\b_k)$ besides universal quantities.
\end{proof}

\section{H\"older Right-Hand Side}

In this section we assume $f$ to be H\"older continuous and establish a similar modulus of continuity for $u_t$. The main idea consists into applying Lemma \ref{thm:improvement} for $\b$ sufficiently close to one followed by some modifications to the proofs of Corollaries \ref{cor:iteration} and \ref{cor:improvement_dif_quot}.

\begin{theorem}\label{thm:improv_reg_2}
Let $F$ be uniformly elliptic satisfying hypothesis \eqref{eq:H} with $F(0,0,x) = 0$ and $u$ satisfies,
\begin{align*}
u_t - F(D^2u,Du,x) = f(x,t) \text{ in the viscosity sense in } Q_1
\end{align*}
Assume that for all $x\in B_2$, $f(x,\cdot) \in C^{0,\gamma/2}[-4,0]$ for some $\gamma \in (0, \bar \a)$ where $\bar \a$ is the exponent from the Krylov-Safonov theorem. Then $u_t$ exists pointwise, and for some constant $C>0$ depending on $\min(\gamma,(\bar\a-\gamma))$,
\[
 \|u_t\|_{C^{0,\gamma}(Q_{1/4})} \leq C\1\sup_{Q_2}|u| + \sup_{x\in B_2}\|f(x,\cdot)\|_{C^{0,\gamma/2}(-4,0]}\2.
\]
\end{theorem}

\begin{proof}
Let us assume without loss of generality that for some $\d\in(0,1)$ sufficiently small,
\[
\sup_{Q_2}|u|\leq 1 \qquad\text{ and }\qquad \sup_{x\in B_2}\|f(x,\cdot)\|_{C^{0,\gamma/2}(-4,0]}\leq \d.
\]
By Theorem \ref{thm:improv_reg_1} we know that for $\b := 1-(\bar\a-\gamma)/4$ there exists some $\e\in(0,1)$ such that,
\[
\sup_{\t\in(0,1/4)} \left\|\frac{\d_\t u}{\t^\b}\right\|_{C^{0,\e}(Q_{1/2})} \leq C.
\]
Then we apply Lemma \ref{thm:improvement} to get the diminishment for $\a = (\bar\a+\gamma)/2$ and some $\m\in(0,1/2)$,
\[
\sup_{\t\in(0,1/4)} \left[\frac{\d_\t u}{\t^\b}\right]_{C^{0,\e}(Q_\m)} \leq C\m^\a.
\]

Then we would like to apply Corollary \ref{cor:iteration}, however we have $\b+\a/2 = 1+\gamma/4$ which is not smaller than one as required in the statement of such corollary. The necessity of such hypothesis appears in the rescaling we considered, so let us recall the setup,
\begin{align*}
v(x,t) &:= \frac{u(\m^i x,\m^{2i}t)}{\m^{2i(\b+\a/2+\e/2)}},\\
\tilde F(M,p,x) &:= \m^{2i(1-\b-\a/2-\e/2)}F(\m^{-2i(1-\b-\a/2-\e/2)} M,\m^{-2i(1/2-\b-\a/2-\e/2)} p,\m^i x),\\
\tilde f(x,t) &:= \m^{2i(1-\b-\a/2-\e/2)}f(\m^i x,\m^{2i} t).
\end{align*}
At this point we can use the H\"older hypothesis for $f$ to obtain,
\[
\sup_{\substack{\t\in(0,1/4)\\(x,t)\in Q_{1/2}}}|\d_\t \tilde f(x,t)| \leq \d\mu^{2i(1-\b-\a/2-\e/2)}\m^{i\gamma} = \d\m^{i(\gamma/2-\e)} \leq 1,
\]
provided that $\e\in(0,\gamma/2)$. In conclusion, the same argument from the proof of Corollary \ref{cor:iteration} applies in order to obtain, after a standard covering argument that,
\begin{align}\label{eq:inth1}
\sup_{\substack{Q_r(x,t)\ss Q_{1/4}\\\t\in(0,r^2)}}\osc_{Q_r(x,t)}\frac{\d_\t u}{r^{\a+\e}\t^{\b}} \leq C.
\end{align}

Now we use Lemma \ref{lem:appendix4} to get the bounds
\begin{equation}\label{eq:inti2}
 \sup_{Q_{1/4}}|u_t| +  \sup_{\substack{(x,t)\in Q_{1/4}\\\t\in(0,1/4)}} \left| \frac{\d_\t u_t(x,t)}{\t^{\gamma/4}}   \right| \leq C.
\end{equation}
At this moment all we have to show is that given $(y,s)\in Q_r(x,t)\ss Q_{1/4}$,
\begin{align}\label{eq:7}
 \left| u_t(x,t) - u_t(y,s)\right| \leq Cr^{\gamma/2}.
\end{align}

Let $\t = r^2$ such that from \eqref{eq:inth1} we obtain,
\[
 \left| \frac{\d_{r^2}u(x,t)}{r^2} - \frac{\d_{r^2}u(y,s)}{r^2}\right| \leq r^{\gamma/2}.
\]
On the other hand, using \eqref{eq:inti2},
\[
\left|\frac{\d_{r^{2}}u}{r^2}-u_t\right|(y,s) \leq \frac{1}{r^2}\int_{-r^2}^0|u_t(y,s+a)-u_t(y,s)|da\leq Cr^{\gamma/2}
\]
Finally the desired estimate results from the triangular inequality by adding and subtracting $\1\frac{\d_{r^2}u(x,t)}{r^2} - \frac{\d_{r^2}u(y,s)}{r^2}\2$ inside the absolute value in \eqref{eq:7}.
\end{proof}

\section{Appendix}

In this appendix we establish a few interpolation results about H\"older spaces. The first lemma can be understood as a maximum principle.

\begin{lemma}
Let $\a,\b\in(0,1)$. Then for any $u \in C[-1,0]$,
\begin{align*}
&u(-1)=u(0)=0, \qquad \sup_{\t\in(0,1)}\left[\frac{\d_{\t}u}{\t^{\b}}\right]_{C^\a[-1+\t,0]} \leq 1\qquad \Rightarrow \qquad \osc_{[-1,0]}u \leq 2.
 \end{align*}
\end{lemma}

\begin{proof}
Let,
\[
\varphi(t) = 2\max\1|t|^{(\a+\b)/4},|t+1|^{(\a+\b)/4}\2.
\]
We want to show that $u \leq \varphi$ in $[-1,0]$ and therefore $u \leq 2$. Assume that there exists $\t\in(0,1/2]$ such that $-\t$ realizes the positive maximum of $(u-\varphi)$ in $(-1,0)$. Then we obtain the following contradiction,
\[
-2\1 2-2^{(\a+\b)/4}\2 \t^{(\a+\b)/4} \geq \d_\t^2\varphi(0) \geq \d_\t^2 u(0) \geq -\t^{\a+\b}.
\]
A similar contradiction happens if $\t\in(1/2,1)$ by considering the second order differences at $-1$.
\end{proof}

The proof of Lemma 5.2 in \cite{Caffarelli95} shows that if $\a+\b<1$ then there exits some constant $C>0$ depending on $1-(\a+\b)$ such that the following estimate holds,
\[
\sup_{\t\in(0,1)}\left|\frac{\d_\t u(0)}{\t^{\a+\b}}\right| \leq C\1\osc_{[-1,0]}u + \sup_{\t\in(0,1)}\left[\frac{\d_{\t}u}{\t^{\b}}\right]_{C^\a[-1+\t,0]}\2.
\]

By applying this result followed by the maximum principle to,
\[
\bar u(s) := \frac{u(\bar\t s) + s u(-\bar\t) - (s+1)u(0)}{\bar\t^{\a+\b}\sup_{\t\in(0,\bar\t)}\left[\frac{\d_\t u}{\t^\b}\right]_{C^\a[\t-\bar\t,0]}}
\]
we get the following corollary.

\begin{corollary}\label{lem:appendix3}
Let $\a,\b\in(0,1)$ such that $\a+\b<1$. There exists a constant $C>0$ depending on $1-(\a+\b)$ such that for any $u \in C[-1,0]$ and $\bar\t\in(0,1)$,
\[
\sup_{\t\in(0,\bar\t)}\left|\frac{\d_\t u(0)}{\t^{\b+\a}}\right| \leq \bar\t \left|\frac{\d_{\bar\t} u(0)}{\bar\t^{\b+\a}}\right| + C\sup_{\t\in(0,1)}\left[\frac{\d_\t u}{\t^\b}\right]_{C^\a[-1+\t,0]}.
\]
In particular,
\[
\sup_{\t\in(\bar\t,1)}\left|\frac{\d_{\t} u(0)}{\t^{\b}}\right| \geq \frac{1}{2}\sup_{\t\in(0,1)}\left|\frac{\d_\t u(0)}{\t^{\b}}\right| - C\bar\t^\a\sup_{\t\in(0,1)}\left[\frac{\d_\t u}{\t^\b}\right]_{C^\a[-1+\t,0]}.
\]
\end{corollary}

Finally, this last lemma establishes a H\"older estimate for the derivative when $\a+\b>1$.

\begin{lemma}\label{lem:appendix4}
Let $\a,\b\in(0,1)$ such that $\a+\b>1$ and $u:[-1,1]\to\R$ such that,
\[
\sup_{\t\in(0,1)}\left\|\frac{\d_\t u}{\t^\b}\right\|_{C^\a(-1+\t,1)} \leq 1.
\]
Then for some universal constant $C$,
\[
\|u_t\|_{C^{1,\a+\b-1}(-1,1)} \leq C.
\]
\end{lemma}

\begin{proof}
By Lemma 5.6 from \cite{Caffarelli95} we know that $u$ is Lipschitz and therefore differentiable almost everywhere. By a density argument it suffices to show that that for each point of differentiability $t_0\in(-1,1)$,
\begin{align*}
|u(t)-u(t_0)-u_t(t_0)(t-t_0)| \leq C|t|^{\a+\b} \text{ for } t \in [-1,1].
\end{align*}

Assume without loss of generality that $t_0 = u(t_0) = u_t(t_0) = 0$. If there exists $h\in(0,1]$ such that $u(h) > Ch^{\a+\b}$, then by iterating the hypothesis of the Lemma we get for every $i\in\N$,
\begin{align*}
\frac{u(2^{-i}h)}{2^{-i}h} > \1C-\sum_{j=0}^{i-1}2^{-(\a+\b-1)j}\2 h^{\a+\b-1} \geq \frac{C}{2}h^{\a+\b-1}>0,
\end{align*}
provided that $C = 4/(2^{\a+\b-1}-1)$. This contradicts $u_t(0)=0$ as $i\to\8$. 
\end{proof}

{\bf Acknowledgments:}
DK was partially supported by National Science Foundation grant DMS-1065926.

\bibliographystyle{plain}
\bibliography{mybibliography}

\end{document}